\documentclass[12pt]{article}

\usepackage{amsthm}
\usepackage{amscd}
\usepackage{amsfonts}
\usepackage{amssymb}
\usepackage{amsgen}
\usepackage{amsmath}
\usepackage{amsopn}
\usepackage{verbatim}
\usepackage{xypic}
\usepackage{xspace}
\usepackage{multicol}
\usepackage{url}
\usepackage{upref}

\theoremstyle{plain}
\newtheorem{thm}{Theorem}[section]
\newtheorem{lem}[thm]{Lemma}
\newtheorem{prop}[thm]{Proposition}
\newtheorem{cor}[thm]{Corollary}

\theoremstyle{definition}

\newtheorem{rem}[thm]{Remark}

 \font\cyr=wncyr10
 \newcommand{\nc}{\newcommand}

\DeclareMathOperator{\Sym}{{Sym}}

\nc{\per}[1]{\underset{#1}{\boldsymbol \pi}\,}

 \nc{\MT}{{\rm MT}}
 \nc{\XX}{{X}}
 \nc{\gF}{{\varPhi}}
 \nc{\ot}{\otimes}

 \nc{\evalE}{{\mathfrak Z}}
 \nc{\tG}{\tilde{G}}
 \nc{\wht}{\widehat}
 \nc{\bwg}{{\bigwedge}}
 \nc{\wg}{{\wedge}}
 \nc{\mmu}{{\boldsymbol{\mu}}}
 \nc{\mal}{{{\scriptstyle \maltese}}}
 \nc{\fA}{{\mathfrak A}}
 \nc{\HH}{{\mathfrak H}}
 \nc{\ra}{\rightarrow}
 \nc{\ors}{{\bfs}}
 \nc{\orr}{{\bfr}}
 \nc{\os}{{\overset}}
 \nc{\G}{{\mathbb G}}
 \nc{\F}{{\mathbb F}}
 \nc{\Z}{{\mathbb Z}}
 \nc{\R}{{\mathbb R}}
 \nc{\N}{{\mathbb N}}
 \nc{\ZN}{{\mathbb Z_{\ge 0}}}
 \nc{\Q}{{\mathbb Q}}
 \nc{\C}{{\mathbb C}}
 \nc{\CP}{{\mathbb{CP}}}
 \nc{\Cnn}{{\mathbb C}_{\ge 0}}
 \nc{\Cp}{{\mathbb C}_{>0}}
 \nc{\MPV}{{\mathcal{MPV}}}
 
 \nc{\tB}{{\tilde B}}
 \nc{\tP}{{P}}
 \nc{\tQ}{{Q}}
\nc{\gemn}{{\mathfrak n}}
 \nc{\suf}{{\ast\,}}
 \nc{\sufq}{{\ast_q\,}}
 \nc{\gam}{{\gamma}}
 \nc{\gG}{{\Gamma}}
 \nc{\om}{{\omega}}
 \nc{\vep}{{\varepsilon}}
 \nc{\ga}{{\alpha}}
 \nc{\gl}{{\lambda}}
 \nc{\gb}{{\beta}}
 \nc{\gf}{{\varphi}}
 \nc{\gd}{{\delta}}
 \nc{\orgd}{{\vec \gd\,}}
 \nc{\gs}{{\sigma}}
 \nc{\gth}{{\theta}}
 \nc{\gx}{{\xi}}
 \nc{\gS}{{\Sigma}}

 \nc{\gk}{{\kappa}}
  \nc{\gz}{{\zeta}}
 \nc{\tgz}{{\tilde{\zeta}}}
 \nc{\gO}{{\Omega}}
 \nc{\sif}{{\mathcal S}}
 \nc{\gt}{{\tau}}
 \nc{\Lra}{\Longrightarrow}
 \nc{\lra}{\longrightarrow}
 \nc{\lmaps}{\longmapsto}
 \nc{\fS}{{\mathfrak S}}
 \nc{\DD}{{\mathfrak D}}
 \nc{\Llra}{\Longleftrightarrow}
 \nc{\ol}{\overline}
 \nc{\ola}{\overleftarrow}
 \nc{\lms}{\longmapsto}
 \nc{\cv}{{{\mathsf c}{\mathsf v}}}
 \nc{\zq}{{\zeta_q}}
 \nc\qup{{q\uparrow 1}}
 \nc{\us}{\underset}
 \nc{\tn}{{\tilde{n}}}
 \nc{\gD}{{\Delta}}
 \nc{\bi}{{\bf i}}
 \nc{\bfone}{{\bf 1}}

 \nc{\bfa}{{\bf a}}
 \nc{\bfb}{{\bf b}}
 \nc{\bfc}{{\bf c}}
 \nc{\bfd}{{\bf d}}
 \nc{\bfe}{{\bf e}}
 \nc{\bff}{{\bf f}}
 \nc{\bfg}{{\bf g}}
 \nc{\bfi}{{\bf i}}
 \nc{\bfj}{{\bf j}}
\nc{\obfi}{{\overrightarrow{\boldsymbol \imath}}}
\nc{\obfj}{{\overrightarrow{\boldsymbol \jmath}}}
\nc{\obfd}{{\overrightarrow{\bf d}}}
\nc{\veps}{{\varepsilon}}
 \nc{\bfn}{{\bf n}}
 \nc{\bfl}{{\bf l}}
 \nc{\bfk}{{\bf k}}
 \nc{\bfm}{{\bf m}}
 \nc{\bfo}{{\bf o}}
 \nc{\bfp}{{\bf p}}
 \nc{\bfq}{{\bf q}}
 \nc{\bfr}{{\bf r}}
 \nc{\bfs}{{\bf s}}
 \nc{\bft}{{\bf t}}
 \nc{\bfu}{{\bf u}}
 \nc{\bfv}{{\bf v}}
 \nc{\bfw}{{\bf w}}
 \nc{\bfx}{{\bf x}}
 \nc{\bfy}{{\bf y}}
 \nc{\bfz}{{\bf z}}
 \nc{\bfB}{{\bf B}}
 \nc{\bfP}{{\bf P}}
 \nc{\bfQ}{{\bf Q}}
 \nc{\bfY}{{\bf Y}}
 \nc{\bfgb}{{\boldsymbol \gb}}
 \nc{\bfga}{{\boldsymbol \ga}}
 \nc{\bfrho}{{\boldsymbol \rho}}
 \nc{\bfchi}{{\boldsymbol \chi}}
 \nc{\QX}{{\Q\langle \bfX\rangle}}
 \nc{\QY}{{\Q\langle \bfY\rangle}}
 \nc{\CX}{{\C\langle \bfX\rangle}}
 \nc{\CY}{{\C\langle \bfY\rangle}}
 \nc{\QXX}{{\Q\langle\!\langle \bfX\rangle\!\rangle}}
 \nc{\QYY}{{\Q\langle\!\langle \bfY\rangle\!\rangle}}
 \nc{\CXX}{{\C\langle\!\langle \bfX\rangle\!\rangle}}
 \nc{\CYY}{{\C\langle\!\langle \bfY\rangle\!\rangle}}

 \nc{\bbA}{{\mathbb A}}
 \nc{\bbB}{{\mathbb B}}
 \nc{\bbC}{{\mathbb C}}
 \nc{\bbD}{{\mathbb D}}
 \nc{\bbE}{{\mathbb E}}
 \nc{\bbF}{{\mathbb F}}
 \nc{\bbG}{{\mathbb G}}
 \nc{\bbH}{{\mathbb H}}
 \nc{\bbI}{{\mathbb I}}
 \nc{\bbJ}{{\mathbb J}}
 \nc{\bbK}{{\mathbb K}}
 \nc{\bbL}{{\mathbb L}}
 \nc{\bbM}{{\mathbb M}}
 \nc{\bbN}{{\mathbb N}}
 \nc{\bbO}{{\mathbb O}}
 \nc{\bbP}{{\mathbb P}}
 \nc{\bbQ}{{\mathbb Q}}
 \nc{\bbR}{{\mathbb R}}
 \nc{\bbS}{{\mathbb S}}
 \nc{\bbT}{{\mathbb T}}
 \nc{\bbU}{{\mathbb U}}
 \nc{\bbV}{{\mathbb V}}
 \nc{\bbW}{{\mathbb W}}
 \nc{\bbX}{{\mathbb X}}
 \nc{\bbY}{{\mathbb Y}}
 \nc{\bbZ}{{\mathbb Z}}
 \nc{\bba}{{\mathbb a}}
 \nc{\bbb}{{\mathbb b}}
 \nc{\bbc}{{\mathbb c}}
 \nc{\bbd}{{\mathbb d}}
 \nc{\bbe}{{\mathbb e}}
 \nc{\bbf}{{\mathbb f}}
 \nc{\bbg}{{\mathbb g}}
 \nc{\bbh}{{\mathbb h}}
 \nc{\bbi}{{\mathbb i}}
 \nc{\bbk}{{\mathbb k}}
 \nc{\bbl}{{\mathbb l}}
 \nc{\bbm}{{\mathbb m}}
 \nc{\bbn}{{\mathbb n}}
 \nc{\bbo}{{\mathbb o}}
 \nc{\bbp}{{\mathbb p}}
 \nc{\bbq}{{\mathbb q}}
 \nc{\bbr}{{\mathbb r}}
 \nc{\bbs}{{\mathbb s}}
 \nc{\bbt}{{\mathbb t}}
 \nc{\bbu}{{\mathbb u}}
 \nc{\bbv}{{\mathbb v}}
 \nc{\bbw}{{\mathbb w}}
 \nc{\bbx}{{\mathbb x}}
 \nc{\bby}{{\mathbb y}}
 \nc{\bbz}{{\mathbb z}}

 \nc{\MZV}{{\mathcal{MZV}}}
 \nc{\calA}{{\mathcal A}}
 \nc{\calB}{{\mathcal B}}
 \nc{\calC}{{\mathcal C}}
 \nc{\calD}{{\mathcal D}}
 \nc{\calE}{{\mathcal E}}
 \nc{\calF}{{\mathcal F}}
 \nc{\calG}{{\mathcal G}}
 \nc{\calH}{{\mathcal H}}
 \nc{\calI}{{\mathcal I}}
 \nc{\calJ}{{\mathcal J}}
 \nc{\calK}{{\mathcal K}}
 \nc{\calL}{{\mathcal L}}
 \nc{\calM}{{\mathcal M}}
 \nc{\calN}{{\mathcal N}}
 \nc{\calO}{{\mathcal O}}
 \nc{\calP}{{\mathcal P}}
 \nc{\calQ}{{\mathcal Q}}
 \nc{\calR}{{\mathcal R}}
 \nc{\calS}{{\mathcal S}}
 \nc{\calT}{{\mathcal T}}
 \nc{\calU}{{\mathcal U}}
 \nc{\calV}{{\mathcal V}}
 \nc{\calW}{{\mathcal W}}
 \nc{\calX}{{\mathcal X}}
 \nc{\calY}{{\mathcal Y}}
 \nc{\calZ}{{\mathcal Z}}
  \nc{\cala}{{\mathcal a}}
 \nc{\calb}{{\mathcal b}}
 \nc{\calc}{{\mathcal c}}
 \nc{\cald}{{\mathcal d}}
 \nc{\cale}{{\mathcal e}}
 \nc{\calf}{{\mathcal f}}
 \nc{\calg}{{\mathcal g}}
 \nc{\calh}{{\mathcal h}}
 \nc{\cali}{{\mathcal i}}
 \nc{\calj}{{\mathcal j}}
 \nc{\calk}{{\mathcal k}}
 \nc{\call}{{\mathcal l}}
 \nc{\calm}{{\mathcal m}}
 \nc{\caln}{{\mathcal n}}
 \nc{\calo}{{\mathcal o}}
 \nc{\calp}{{\mathsf p}}
 \nc{\calq}{{\mathcal q}}
 \nc{\calr}{{\mathcal r}}
 \nc{\cals}{{\mathcal s}}
 \nc{\calt}{{\mathcal t}}
 \nc{\calu}{{\mathcal u}}
 \nc{\calv}{{\mathcal v}}
 \nc{\calw}{{\mathcal w}}
 \nc{\calx}{{\mathcal x}}
 \nc{\caly}{{\mathcal y}}
 \nc{\calz}{{\mathcal z}}

 \nc{\frakA}{{\mathfrak A}}
 \nc{\frakB}{{\mathfrak B}}
 \nc{\frakC}{{\mathfrak C}}
 \nc{\frakD}{{\mathfrak D}}
 \nc{\frakE}{{\mathfrak E}}
 \nc{\frakF}{{\mathfrak F}}
 \nc{\frakG}{{\mathfrak G}}
 \nc{\frakH}{{\mathfrak H}}
 \nc{\frakI}{{\mathfrak I}}
 \nc{\frakJ}{{\mathfrak J}}
 \nc{\frakK}{{\mathfrak K}}
 \nc{\frakL}{{\mathfrak L}}
 \nc{\frakM}{{\mathfrak M}}
 \nc{\frakN}{{\mathfrak N}}
 \nc{\frakO}{{\mathfrak O}}
 \nc{\frakP}{{\mathfrak P}}
 \nc{\frakQ}{{\mathfrak Q}}
 \nc{\frakR}{{\mathfrak R}}
 \nc{\frakS}{{\mathfrak S}}
 \nc{\frakT}{{\mathfrak T}}
 \nc{\frakU}{{\mathfrak U}}
 \nc{\frakV}{{\mathfrak V}}
 \nc{\frakW}{{\mathfrak W}}
 \nc{\frakX}{{\mathfrak X}}
 \nc{\frakY}{{\mathfrak Y}}
 \nc{\frakZ}{{\mathfrak Z}}
 \nc{\fraka}{{\mathfrak a}}
 \nc{\frakb}{{\mathfrak b}}
 \nc{\frakc}{{\mathfrak c}}
 \nc{\frakd}{{\mathfrak d}}
 \nc{\frake}{{\mathfrak e}}
 \nc{\frakf}{{\mathfrak f}}
 \nc{\frakg}{{\mathfrak g}}
 \nc{\frakh}{{\mathfrak h}}
 \nc{\fraki}{{\mathfrak i}}
 \nc{\frakj}{{\mathfrak j}}
 \nc{\frakk}{{\mathfrak k}}
 \nc{\frakl}{{\mathfrak l}}
 \nc{\frakm}{{\mathfrak m}}
 \nc{\frakn}{{\mathfrak n}}
 \nc{\frako}{{\mathfrak o}}
 \nc{\frakp}{{\mathfrak p}}
 \nc{\frakq}{{\mathfrak q}}
 \nc{\frakr}{{\mathfrak r}}
 \nc{\fraks}{{\mathfrak s}}
 \nc{\frakt}{{\mathfrak t}}
 \nc{\fraku}{{\mathfrak u}}
 \nc{\frakv}{{\mathfrak v}}
 \nc{\frakw}{{\mathfrak w}}
 \nc{\frakx}{{\mathfrak x}}
 \nc{\fraky}{{\mathfrak y}}
 \nc{\frakz}{{\mathfrak z}}
 \nc{\so}{{\mathfrak so}}
 \nc{\sa}{{\mbox{{\scriptsize \cyr x}}}}
 \nc{\slfour}{{\mathfrak sl}_4}
 \nc{\one}{{\bf 1}}
 \nc{\zero}{{\bf 0}}
 \nc{\Qxy}{\Q\langle x,y\rangle}

 \nc{\barN}{{\ol{\mathbb N}}}

\begin{document}
\title{Restricted Sum Formula of Alternating Euler Sums}
\author{Jianqiang Zhao}
\date{}
\maketitle

\begin{center}
 {Department of Mathematics, Eckerd College, St. Petersburg, FL 33711}
\end{center}

\begin{abstract}
In this paper we study restricted sum formulas involving alternating Euler sums
which are defined by
\begin{equation*}
    \gz(s_1,\dots,s_d;\vep_1,\dots,\vep_d)=\sum_{n_1>\cdots>n_d\ge 1}
    \frac{\vep_1^{n_1}\cdots \vep_d^{n_d}}{ n_1^{s_1}\cdots n_d^{s_d}},
\end{equation*}
for all positive integers $s_1,\dots,s_d$ and  $\vep_1=\pm 1,\dots, \vep_d=\pm 1$
with $(s_1,\vep_1)\ne (1,1)$. We call $w=s_1+\cdots+s_d$ the weight and $d$ the depth.
When $\vep_j=-1$ we say the $j$th component is alternating.
We first consider Euler sums of the following special type:
\begin{equation*}
    \gx(2s_1,\dots,2s_d)=\gz(2s_1,\dots,2s_d;(-1)^{s_1},\dots,(-1)^{s_d}).
\end{equation*}
For $d\le n$, let $\Xi(2n,d)$ be the sum of all $\gx(2s_1,\dots, 2s_d)$
of fixed weight $2n$ and depth $d$. We derive a formula for $\Xi(2n,d)$
using the theory of symmetric functions established by Hoffman recently.
We also consider restricted sum formulas of Euler sums with fixed weight $2n$,
depth $d$ and fixed number $\ga$ of alternating components at even arguments. 
When $\ga=1$ or $\ga=d$
we can determine precisely the restricted sum formulas. For other $\ga$ 
we only treat the cases $d<5$ completely since the symmetric function 
theory becomes more and more unwieldy to work with when $\ga$ moves 
closer to $d/2$.
\end{abstract}

\section{Introduction}
The multiple zeta values (MZVs) $\gz(s_1,\dots,s_d)$ of depth $d$ are defined
by the iterated sum
\begin{equation*}
\gz(s_1,\dots,s_d)=\sum_{ k_1>\dots>k_d>0} k_1^{-s_1}\cdots k_d^{-s_d},
\end{equation*}
for all positive integers $s_1,\dots,s_d$ with $s_1\ge 2$. One of the
central problems on MZVs is to determine all the possible 
$\Q$-linear relations among them. Standard conjectures
in arithmetic algebraic geometry (see for e.g. \cite[Conjecture 4.2]{Goncharov1997})
imply that we only need to study such relations among MZVs of the same weight.
Such relations are commonly believed to be studied systematically
by Euler for double zeta functions while many new ones have been found
in recent years. Gangl, Kaneko and Zagier \cite{GKZ}  showed that
\begin{equation}\label{equ:GKZ}
\sum_{a,b>0, a+b=n}\gz(2a,2b)=\frac34\gz(2n).
\end{equation}
Such formulas with arguments running through only even integers
are called restricted sum formulas.
Shen and Cai \cite{ShenCai2012} generalized \eqref{equ:GKZ} to
multiple zeta values of fixed depth $d\le 4$ and fixed weight $2n$.
Hoffman \cite{Hoffman2012} recently extended this to arbitrary depths.
In \cite{ZTn}, the author obtained a similar result for multiple Hurwitz-zeta values
whose small depth cases were first studied by Shen and Cai \cite{ShenCai2011}.

Higher level analogs of MZVs are given by special values of multiple polylogarithms
at $N$-th roots of unity (see \cite{DeligneGo2005,Zhao2010b}). 
These values are complex values in
general when $N>2$. In level $N=1$ we recover MZVs and in level $N=2$ we obtain
(alternating) Euler sums which are defined as follows.
For all positive integers $s_1,\dots,s_d$ and $\vep_1=\pm 1,\dots, \vep_d=\pm 1$
with $(s_1,\vep_1)\ne (1,1)$ 
\begin{equation*}
    \gz(s_1,\dots,s_d;\vep_1,\dots,\vep_d)=\sum_{n_1>\cdots>n_d\ge 1}
    \frac{\vep_1^{n_1}\cdots \vep_d^{n_d}}{ n_1^{s_1}\cdots n_d^{s_d}}.
\end{equation*}
Following the convention we call $w=s_1+\cdots+s_d$ the weight and $d$ the depth. 
If $\vep_j=-1$ for some $j$ we say the $j$th component is alternating
and put a bar over $s_j$ by which we can suppress $\vep_j$. 
For example, it is fairly easy to see that
\begin{equation}\label{equ:olzeta}
\gz(\ol{n})+\gz(n)=\sum_{j=1}^\infty \frac{(-1)^n+1}{j^n}=2^{1-n}\gz(n) \Longrightarrow\gz(\ol{n})=(2^{1-n}-1)\gz(n).
\end{equation}

We shall derive a few different types of restricted sum formulas for
alternating Euler sums. First we treat a very special type.
For all positive integers $s_1,\dots,s_d$ we set
\begin{align*}
    \gx(2s_1,\dots,2s_d)=&\gz(2s_1,\dots,2s_d;(-1)^{s_1},\dots,(-1)^{s_d})\\
    =&\sum_{n_1>\cdots>n_d\ge 1}
    \frac{(-1)^{s_1n_1}\cdots (-1)^{s_dn_d}}{ n_1^{2s_1}\cdots n_d^{2s_d}}.
\end{align*}
For example,
\begin{equation*}
    \gx(\{-2\}^d)=\gz(\{\ol{2}\}^d), \quad \gx(\{-4\}^d)=\gz(\{4\}^d).
\end{equation*}
Here for any string $S$ we denote by $S^d$ the string obtained by repeating $S$
exactly $d$ times. Put
\begin{equation*}
\Xi(2n,d)=\sum_{\substack{j_1+\cdots+j_d=n\\ j_1,\dots,j_d > 0}}
\gx(2j_1,\dots,2j_d).
\end{equation*}
For example,
$$\Xi(6,2)=\gx(2,4)+\gx(4,2)=\gz(\ol{2},4)+\gz(4,\ol{2}).$$
It turns out that the formula for $\Xi(2n,d)$ is more complicated
than either of the two in \cite{Hoffman2012,ZTn} for MZVs and multiple
Hurwitz-zeta values respecitively.
In fact, in a sense it combines the two.

To state our main results concerning $\Xi(2n,d)$ we recall that the Bernoulli
numbers $B_j$ and Euler numbers $E_j$
are defined by the following generating functions respectively:
\begin{equation} \label{equ:EulerNumber}
 \frac{x}{e^x-1}=\sum_{j=0}^\infty \frac{B_j}{j!}x^j, \qquad
 \sec x=\sum_{j=0}^\infty (-1)^{j} \frac{E_{2j}}{ (2j)!}x^{2j},
\end{equation}
and $E_{2j+1}=0$ for all $j\ge 0$.
\begin{thm} \label{thm:Xi2nd}
For all positive integers $d\le n$,
\begin{align*}
\Xi(2n,d)=& \sum_{j=0}^{\lfloor\frac{d-1}{2}\rfloor}
 \frac{(-1)^{\lfloor\frac{j}{2}\rfloor}\pi^{2j}\tilde{\gz}_n(2n-2j)}{2^{2d-j-2}(2j+1)!} \binom{2d-2j-1}{d} ,\\
+&\sum_{j=0}^{\lfloor\frac{d-2}4\rfloor}  \left(\sum_{\substack{r+s=n-2j\\ r,s\ge 0}}
\frac{(-1)^{r} (4^r-1)}{4^{n-2j}}  \gz(2r)  \gz(2s) \right)
\frac{(-1)^j \pi^{4j} }{2^{2d-2j-5}(4j+2)!}
\binom{2d-4j-2}{d},
\end{align*}
where $\tilde{\gz}_n(m)=\gz(m)$ if $n$ is even and 
$\tilde{\gz}_n(m)=\gz(\ol{m})=(2^{1-m}-1)\gz(m)$ if $n$ is odd.
In particular, $\Xi(2n,d)$ is a rational multiple of $\pi^{2n}$.
\end{thm}

For example,
\begin{align*}
\Xi(2n,2)=&\frac{3\tilde{\gz}_n(2n)}{4}
+ \sum_{\substack{r+s=n\\ r,s\ge 0}}
\frac{(-1)^{r} (4^r-1)}{4^n}  \gz(2r)  \gz(2s),\\
\Xi(2n,3)=&\frac{5\tilde{\gz}_n(2n)}{8}+\frac{\gz(2)\tilde{\gz}_n(2n-2)}{8}
+ \sum_{\substack{r+s=n\\ r,s\ge 0}}
\frac{(-1)^{r} (4^r-1)}{2^{2n-1}}  \gz(2r)  \gz(2s) . 
\end{align*}
The key to the proof of Theorem~\ref{thm:Xi2nd} is to study the
generating function of $\Xi(2n,d)$  
$$\phi(u,v) =1+\sum_{n\ge d\ge 1} \Xi(2n,d)u^n v^d$$
for which we have the following result.
\begin{thm} \label{thm:gfun}
We have
\begin{equation*}
\phi(u,v)= \frac{ \sin(\pi\sqrt{(1-v)u}/2)}{ \sqrt{1-v} \sin(\pi\sqrt{u}/2)}
\frac{ \cosh(\pi\sqrt{(1-v)u}/2)}{ \cosh(\pi\sqrt{u}/2)}.
\end{equation*}
\end{thm}

Another formula for $\Xi(2n,d)$ 
is more useful computationally when $d$ is close to $n$.
\begin{thm} \label{thm:longform}
For $d\le n$ we have
\begin{equation}  \label{equ:longform}
\Xi(2n,d)= \sum_{\ell =0}^{n-d}\binom{n-\ell}{d}
\frac{2(-1)^{d+\lfloor (\ell-n-1)/2\rfloor} }{2^{n+\ell}(2n-2\ell+1)!} \sum_{j=0}^\ell \gz(\ol{2\ell-2j})\frac{E_{2j}}{(2j)!}\pi^{2n-2\ell+2j}.
\end{equation}
\end{thm}

The second type restricted sum formula for Euler sums has the form
\begin{equation*}
A_\ga(2n,d) :=\sum_{\substack{j_1+\cdots+j_d=n,\,j_1,\dots,j_d > 0
\\ \sharp\{\ell:\  \vep_\ell=-1\}=\ga}}
 \gz(s_1,\dots,s_d;\vep_1,\dots,\vep_d).
\end{equation*}
Note that $\ga$ counts the number of alternating components.
When $\ga=1$ or $\ga=d$ we shall determine the restricted sum formulas
by using the theory of symmetric functions. 
\begin{thm}\label{thm:Aga=1dArbitrayIntro}
For every positive integers $d\le n$ we have
\begin{equation}\label{equ:Aga=1dArbitray}
A_1(2n,d)=\gz(\ol{2n})-\sum_{j=0}^{\lfloor(d-2)/2 \rfloor}\sum_{k=0}^{2j+1} \frac{(-1)^{d+j+k}}{(2j+1)!}\binom{2j+1}{k}
\binom{(k-3)/2}{d-1} \gz(2n-2j)\pi^{2j}.
\end{equation}
\end{thm}
The formula for $A_d(2n,d)$ is too complicated to write down here (see Theorem~\ref{thm:Aga=ddArbitray}). These two formulas are derived by
using the following information on their generating functions 
(see Theorem~\ref{thm:gfunTot} and Theorem~\ref{thm:genFunPsi1}).
\begin{thm}\label{thm:generatingFucIntro}
We have
\begin{align*}
\sum_{n\ge d\ge 1} A_1(2n,d) v^d u^n
=&\frac{v\sin(\pi \sqrt{(1-v)u})}{ \sqrt{1-v} \sin(\pi \sqrt{u})}
\left(\frac{1}{2(1-v)}- \frac{\pi\sqrt{u}}{2\sqrt{1-v} \sin( \pi\sqrt{(1-v)u})} \right)\\
1+\sum_{n\ge d\ge 1} A_d(2n,d) v^d u^n 
= &\frac{\sin( \pi \sqrt{(1-v)u}/2) }{ \sqrt{1-v}\sin(\pi\sqrt u/2)}
 \cdot\frac{\cos(\pi \sqrt{(v+1)u}/2)}{ \cos(\pi\sqrt u/2)} .
\end{align*}
\end{thm}
We are able to compute $A_\ga(2n,d)$ for other $\ga$'s when the depth $d<5$ 
by using the stuffle relations (also called quasi-shuffle or harmonic shuffle relations)
satisfied by the alternating Euler sums. This method
can be pushed to larger depth cases by brutal force but it is clearly not
the ideal approach. It seems that the symmetric function theory becomes 
harder and harder to apply when $\ga$ moves closer and closer to $d/2$.
As a result, to find a computationally simple formula for arbitrarily
fixed $n$, $d$ and $\ga$ might require some new ideas.

This work was started while the the author was visiting Taida Institute for
Mathematical Sciences at National Taiwan University in the summer
of 2012. He would like to thank Prof. Jing Yu and Chieh-Yu Chang for
encouragement and their interest in his work. This work was partially supported
by NSF grant DMS-1162116.

\section{Proof of Theorem \ref{thm:gfun} and Theorem \ref{thm:longform}}
We first recall some results on symmetric functions contained in 
\cite{Hoffman2012,Macdonald} with some modification.
Let $\Sym$ be the subring of $\Q[\![x_1,x_2,\dots]\!]$ consisting
of the formal power series of bounded degree that are invariant
under permutations of the $x_j$. Define elements $e_j$, $h_j$, and $p_j$ in 
$\Sym$ by the generating functions (see \cite[\S 1.2]{Macdonald})
\begin{align*}
E(u)&=\sum_{j=0}^\infty e_ju^j=\prod_{j=1}^\infty (1+ux_j),\\
H(u)&=\sum_{j=0}^\infty h_ju^j=\prod_{j=1}^\infty \frac1{1-ux_j} = E(-u)^{-1},\\
P(u)&=\sum_{j=1}^\infty p_ju^{j-1}=\sum_{j=1}^\infty \frac{x_j}{1-ux_j}=
\frac{H'(u)}{H(u)} .
\end{align*}
Define an evaluation map (as a ring homomorphism) $\calE:\Sym\to\R$ such that
$\calE(x_j)=(-1)^j/j^2$ for all $i\ge 1$. Hence for all $n\ge 1$
$$\calE(p_n)=\gx(2n)=\sum_{j\ge 1}\frac{(-1)^{jn}}{j^{2n}}.$$

First we need a simple lemma.
\begin{lem} \label{lem:Tnn}
For any positive integer $n$ we have
\begin{equation}\label{equ:Tnn}
\gx(\{2\}^n)=\gz(\{\ol{2}\}^n)=\frac{(-1)^{\lfloor (n+1)/2\rfloor}\pi^{2n}}{2^n(2n+1)!}.
\end{equation}
\end{lem}
\begin{proof}
It is easy to see that
 \begin{align*}
1+\sum_{n=1}^\infty \gz(\{\ol{2}\}^n) x^{2n}=&
\prod_{j=1}^{\infty} \left (1+ \frac{x^2}{(2j)^2} \right)
\prod_{j=1}^{\infty} \left (1-\frac{x^2}{(2j-1)^2} \right)  \\
=& \prod_{j=1}^{\infty} \left (1+ \frac{(x/2)^2}{j^2} \right)
\prod_{j=1}^{\infty} \left (1-\frac{x^2}{j^2} \right) /
\prod_{j=1}^{\infty} \left (1-\frac{x^2}{(2j)^2} \right)   \\
=& \frac{\sinh(\pi x/2)}{\pi x/2} \cdot  \frac{\sin(\pi x)}{\pi x}
\cdot \frac{\pi x/2} {\sin(\pi x/2)}\\
=& \frac{\sinh(\pi x/2)}{\pi x/2} \cos(\pi x/2) \\
=& \frac{e^{(1+i)y}-e^{-(1+i)y}+e^{(1-i)y}-e^{-(1-i)y} }{4y} \quad (\text{set }y=\pi x/2)\\
=& \frac12\sum_{n=0}^\infty  \frac{(1+i)^{2n+1}+(1-i)^{2n+1}}{(2n+1)!} y^{2n} .
 \end{align*}
Hence the lemma follows from a simple computation using the fact that $(1+i)^2=2i$
and $(1-i)^2=-2i$.
\end{proof}

Now let
$N_{n,d}$ be the sum of all the monomial symmetric functions
corresponding to partitions of $n$ having length $d$.
Then clearly
$$\calE(N_{n,d})=\Xi(2n,d).$$
As in \cite{Hoffman2012} we may define
\begin{equation*}
     \calF(u,v)=1+\sum_{n\ge d\ge 1}N_{n,d}u^nv^d ,
\end{equation*}
then $\calE$ sends $\calF(u,v)$ to the generating function
$$\phi(u,v) =1+\sum_{n\ge d\ge 1} \Xi(2n,d)u^nv^d .$$
By Lemma \ref{lem:Tnn} we have
\begin{equation*}
\calE(E(4u^2))= \frac{\sinh(\pi u)}{\pi u} \cos(\pi u) ,
\end{equation*}
and
\begin{equation*}
\calE(H(4u^2))=\calE(E(-4u^2)^{-1})=\frac{\pi u}{\sin(\pi u)} \frac{1}{\cosh(\pi u)} .
\end{equation*}
Thus by \cite[Lemma 1]{Hoffman2012} $\calF(u,v)=E((v-1)u)H(u)$ and we get
\begin{align*}
\phi(4u^2,v)=\calE(E((v-1)4u^2)H(4u^2))
=&\frac{\sin(\pi u\sqrt{1-v})}{ \sqrt{1-v} \sin(\pi u)}
 \frac{\cosh(\pi u\sqrt{1-v} )}{\cosh(\pi u)} .
\end{align*}
This proves Theorem \ref{thm:gfun}.

Setting $v=1$ in Theorem \ref{thm:gfun} we obtain
\begin{align*}
\phi(u^2,1)=&\frac{\pi u/2}{\sin(\pi u/2)} \frac{1}{\cosh(\pi u/2)}\\
=&\sum_{j=0}^\infty \frac{(-1)^j(2-2^{2j})B_{2j}}{(2j)!} \left(\frac{\pi u}{2}\right)^{2j}
\sum_{\ell=0}^\infty \frac{E_{2\ell}}{(2\ell)!} \left(\frac{\pi u}{2}\right) ^{2\ell}.
\end{align*}
Extracting the coefficient of $u^{2n}$ yields immediately the following identity by \eqref{equ:EulerNumber}
\begin{align}
\calE(h_n)=\sum_{d=1}^n \Xi(2n,d)
=&\frac{\pi^{2n}}{4^n (2n)!} \sum_{j=0}^n (-1)^j(2-2^{2j})B_{2j}E_{2n-2j}\binom{2n}{2j} \notag\\
=&\frac{-2}{4^n} \sum_{j=0}^n \gz(\ol{2n-2j})\frac{E_{2j}}{(2j)!} \pi^{2j} .\label{equ:Thn}
\end{align}
Here $\gz(\ol{0})=-1/2$. Now by \cite[Lemma 2]{Hoffman2012} we have
\begin{equation*}
N_{n,d}=\sum_{\ell =0}^{n-d}\binom{n-\ell}{d}(-1)^{n-d-\ell}h_\ell e_{n-\ell}.
\end{equation*}
Applying the homomorphism $\calE$ and using equation \eqref{equ:Tnn} and
\eqref{equ:Thn} we get Theorem~\ref{thm:longform} immediately.

\section{Proof of Theorems \ref{thm:Xi2nd}}
Define $g(y)=\sin y\cosh y/y$. First we want to study its power series expansion.
By definition
\begin{align*}
  g(y)=& \frac{e^{(1+i)y}-e^{-(1+i)y}+e^{(i-1)y}-e^{-(i-1)y} }{4iy}  \\
=& \frac1{2i}\sum_{n=0}^\infty  \frac{(1+i)^{2n+1}+(i-1)^{2n+1}}{(2n+1)!} y^{2n} \\
=&  \sum_{n=0}^\infty  \frac{(1+i)(2i)^{n-1}+(i-1)(2i)^{n-1}}{(2n+1)!} y^{2n}\\
=&  \sum_{n=0}^\infty  \frac{(-1)^{\lfloor n/2\rfloor} 2^n}{(2n+1)!} y^{2n}.
\end{align*}
We get by Theorem \ref{thm:gfun}:
\begin{equation*}
\phi(4u^2,v):= \sum_{d\ge 0} v^d \tG_d(u)=\frac{g(\pi u \sqrt{1-v})}{g(\pi u)}
=\frac{1}{g(\pi u)} \sum_{j=0}^\infty
\frac{(-1)^{\lfloor j/2\rfloor} 2^j(\pi u)^{2j} }{(2j+1)!} (1-v)^j   .
\end{equation*}
Now set $x=(\pi u)^2$ and let $f(x)=g(\sqrt{x})=(\sin\sqrt{x}\cosh\sqrt{x})/\sqrt{x}$.
Let $D$ be the differential operator with respect to $x$. Then
\begin{align*}
 \tG_d(u)=&(-1)^d \frac{1}{f(x)}
\sum_{j\ge d}\frac{(-1)^{\lfloor j/2\rfloor} 2^j x^j}{(2j+1)!}\binom{j}{d} \\
=&  \frac{1}{f(x)}\cdot \frac{(-x)^d}{d!}\cdot
D^d \sum_{j\ge d}\frac{(-1)^{\lfloor j/2\rfloor} 2^j  x^j}{(2j+1)!}  \\
=& \frac{1}{f(x)}\cdot \frac{(-x)^d}{d!}\cdot D^d f(x).
\end{align*}

\begin{lem} \label{lem:solveDE}
For all $d\ge 0$ define the polynomials
\begin{align*}
X_d(x)=&\sum_{j=0}^{\lfloor\frac{d-1}{2}\rfloor} (-1)^{\lfloor\frac{j}{2}\rfloor-1}
  \frac{(8x)^j}{2^{2d-1}(2j+1)!} \binom{2d-2j-1}{d}, \\
Y_d(x)=&\sum_{j=0}^{\lfloor\frac{d-1}{2}\rfloor} (-1)^{\lfloor\frac{j-1}{2}\rfloor}
  \frac{(8x)^j}{2^{2d-1}(2j+1)!} \binom{2d-2j-1}{d}, \\
Z_d(x)=&\sum_{j=0}^{\lfloor\frac{d-2}4\rfloor}(-1)^j \frac{(8x)^{2j+1}}{2^{2d}(4j+2)!}
\binom{2d-4j-2}{d}, \\
W_d(x)=&\sum_{j=0}^{\lfloor\frac{d}4\rfloor}(-1)^j \frac{(8x)^{2j}}{2^{2d}(4j)!}
\binom{2d-4j}{d} .
\end{align*}
Then we have
$$ \tG_d(\sqrt{x}/\pi)=X_d(x)\sqrt{x}\cot\sqrt{x}+Y_d(x)\sqrt{x}\tanh\sqrt{x}+Z_d(x) \cot\sqrt{x}\tanh\sqrt{x}
+W_d(x).$$
\end{lem}
\begin{proof}
It suffices to prove that for all $d\ge 0$
\begin{align*}
f^{(d)}(x)
=&(-1)^d d! x^{-d} X_d(x) \cos\sqrt{x}\cosh\sqrt{x} \\
+&(-1)^d d! x^{-d} Y_d(x) \sin\sqrt{x}\sinh\sqrt{x} \\
+&(-1)^d d! x^{-d-1/2} Z_d(x) \cos\sqrt{x}\sinh\sqrt{x} \\
+&(-1)^d d! x^{-d-1/2} W_d(x) \sin\sqrt{x}\cosh\sqrt{x}.
\end{align*}
Differentiating once we see that we need to show $X_d,$ $Y_d,$ $Z_d$ and $W_d$
are the unique solution to the recursive system of differential equations:
\begin{align*}
 (d+1)X_{d+1}(x)=& dX_d(x)-xX'_d(x)- \frac12 Z_d(x)- \frac12 W_d(x),\\
 (d+1)Y_{d+1}(x)=& dY_d(x)-xY'_d(x)+\frac12 Z_d(x)- \frac12 W_d(x),\\
 (d+1)Z_{d+1}(x)=& \frac{2d+1}2 Z_d(x)-xZ'_d(x)-\frac{x}2 X_d(x)- \frac{x}2 Y_d(x),\\
 (d+1)W_{d+1}(x)=& \frac{2d+1}2 W_d(x)-xW'_d(x)+\frac{x}2 X_d(x)- \frac{x}2 Y_d(x),
\end{align*}
with the initial conditions $X_0(x)=Y_0(x)=Z_0(x)=0$ and $W_0(x)=1$.
This can be verified easily.
\end{proof}

Set $x=\pi^2 u/4$. We have the following well-known power series expansions
\begin{align}\label{equ:cotFormula}
\sqrt{x}\cot \sqrt{x}=& -2\sum_{m=0}^\infty \frac{\gz(2m)}{4^m} u^m, \quad(\gz(0)=-1/2), \\
\sqrt{x}\tanh \sqrt{x}=&2\sum_{m=0}^\infty (-1)^{m-1} (4^m-1) \frac{\gz(2m)}{4^m} u^m, \notag\\
x \cot \sqrt{x}\tanh \sqrt{x}=&
\sum_{m=0}^\infty \left(\sum_{\substack{r+s=m\\ r,s\ge 0}}
(-1)^{r} (4^r-1)  \gz(2r)  \gz(2s) \right) \frac{ 4 u^m }{4^m}.  \notag
\end{align}
Thus
\begin{align*}
 \sqrt{x}\big(\cot \sqrt{x}+\tanh \sqrt{x}\big)= &
-2\sum_{m=0}^\infty  \gz(4m) u^{2m}
-2\sum_{m=0}^\infty  \gz(\ol{4m+2})  u^{2m+1} ,\\
\sqrt{x}\big(\cot \sqrt{x}-\tanh \sqrt{x}\big) = &
-2\sum_{m=0}^\infty  \gz(\ol{4m}) u^{2m}
-2\sum_{m=0}^\infty  \gz(4m+2)  u^{2m+1} .
\end{align*}
Hence
\begin{align*}
 \tG_d\Big(\frac{\sqrt{u}}{2}\Big)=&\left(\sum_{m=0}^\infty  \gz(4m) u^{2m}
+\sum_{m=0}^\infty  \gz(\ol{4m+2})  u^{2m+1}\right)
\sum_{j=0}^{\lfloor\frac{d-1}{4}\rfloor}
  \frac{(-1)^j\pi^{4j}u^{2j}}{2^{2d-2j-2}(4j+1)!} \binom{2d-4j-1}{d}\\
+&\left(\sum_{m=0}^\infty  \gz(\ol{4m}) u^{2m}
+\sum_{m=0}^\infty  \gz(4m+2)  u^{2m+1}\right)
\sum_{j=0}^{\lfloor\frac{d-3}{4}\rfloor}
  \frac{(-1)^j \pi^{4j+2}u^{2j+1}}{2^{2d-2j-3}(4j+3)!} \binom{2d-4j-3}{d} \\
+&\sum_{m=0}^\infty \left(\sum_{\substack{r+s=m\\ r,s\ge 0}}
(-1)^{r} (4^r-1)  \gz(2r)  \gz(2s) \right) \frac{u^m}{4^m}
\sum_{j=0}^{\lfloor\frac{d-2}4\rfloor}\frac{(-1)^j \pi^{4j}u^{2j}}{2^{2d-2j-5}(4j+2)!}
\binom{2d-4j-2}{d}  \\
+&\text{terms of degree $<d$}.
\end{align*}
We can complete the proof of Theorem \ref{thm:Xi2nd} by extracting the
coefficient of $u^n$ from the above.

We would like to point out that the quantity
\begin{equation*}
R_m(2)-R_m(1)=\sum_{\substack{r+s=m\\ r,s\ge 0}}
(-1)^{r} (4^r-1)  \gz(2r)  \gz(2s)
\end{equation*}
does not seem to simplify. Here, for any real $\ga>0$
\begin{equation*}
 R_m(\ga)=\sum_{\substack{r+s=m\\ r,s\ge 0}}
(-\ga^2)^{r}   \gz(2r)  \gz(2s).
\end{equation*}
A formula involving this quantity was written down first
by Ramanujan in his famous notebook \cite[p.\ 276]{Berndt}.
This was proved later by Grosswald \cite{Gro}. The formula says
\begin{multline*}
\ga^{-n}2^{2n} \pi^{2n-1}\sum_{\substack{r+s=n\\ r,s\ge 0}}
(-\ga^2)^{r}  \frac{B_{2r}}{2(2r)!} \frac{B_{2s}}{2(2s)!}\\
=(-\ga)^{1-n}\left(\frac12\gz(2n-1)+
\sum_{k=1}^\infty \frac{k^{1-2n}}{e^{2k\pi\ga }-1}\right)
-\ga^{n-1}\left(\frac12\gz(2n-1)+
\sum_{k=1}^\infty \frac{k^{1-2n}}{e^{2k\pi/\ga}-1}\right).
\end{multline*}
Multiplying both sides by $(-\ga)^n\pi$ we get
\begin{multline*}
R_n(\ga) = \pi \ga \left[(-\ga^2)^{n-1}\left(\frac12\gz(2n-1)+
\sum_{k=1}^\infty \frac{k^{1-2n}}{e^{2k\pi/\ga}-1}\right)
-\left(\frac12\gz(2n-1)+
\sum_{k=1}^\infty \frac{k^{1-2n}}{e^{2k\pi\ga}-1}\right)\right].
\end{multline*}
Clearly if $n$ is odd then $R_n(1)=0$. When $n$ is even we get
\begin{equation*}
R_n(1) = -\pi \left(\gz(2n-1)+
2 \sum_{k=1}^\infty \frac{k^{1-2n}}{e^{2k\pi}-1}\right).
\end{equation*}
Notice the series on the right hand side converges very fast. This fact
was employed by Ramanujan in his
computation of Riemann zeta function at odd integers.

\section{Restricted sum formula of Euler sums with fixed number of
alternating components}
For any integers $0\le \ga\le d\le n$ we consider the sum
\begin{equation*}
A_\ga(2n,d) :=\sum_{\substack{j_1+\cdots+j_d=n,\,j_1,\dots,j_d > 0
\\ \sharp\{\ell:\  \vep_\ell=-1\}=\ga}}
 \gz(s_1,\dots,s_d;\vep_1,\dots,\vep_d).
\end{equation*}
Notice that $\ga$ counts the number of components that are truly alternating. So
in particular $A_0(2n,d)$ is the restricted sum of multiple zeta values of fixed
depth $d$ and weight $2n$ which was studied in \cite{Hoffman2012}. Moreover
\begin{equation*}
A(2n,d)=\sum_{\ga=0}^d A_\ga(2n,d)
\end{equation*}
is the restricted sum of all alternating Euler sums of depth $d$
and weight $2n$. An easy consequence of Hoffman's result is the following.
\begin{prop}\label{prop:Aga=2nd}
For all positive integers $n$ we have
\begin{equation}\label{equ:Aga=2nd}
A(2n,d)=\binom{2d-1}{d}\frac{\gz(2n)}{2^{2n+d-2}}
-\sum_{j=1}^{\lfloor \frac{d-1}2\rfloor}
\binom{2d-2j-1}{d}\frac{\gz(2j)\gz(2n-2j) }{2^{2n+d-3}(2j+1)B_{2j}}.
\end{equation}
\end{prop}
\begin{proof}
Notice that
$$\sum_{n_1>\cdots>n_d\ge 1}
    \frac{ (1+(-1)^{n_1})\cdots (1+(-1)^{n_d})}{n_1^{2s_1}\cdots n_d^{2s_d}}
=\frac{2^d}{2^{2n}} \sum_{n_1>\cdots>n_d\ge 1}
    \frac{1}{n_1^{2s_1}\cdots n_d^{2s_d}}.$$
So \eqref{equ:Aga=2nd} follows immediately from the following formula
\begin{equation}\label{equ:Hoffman}
A_0(2n,d)=\binom{2d-1}{d} \frac{\gz(2n)}{2^{2(d-1)}}
-\sum_{j=1}^{\lfloor \frac{d-1}2\rfloor}
\binom{2d-2j-1}{d}\frac{\gz(2j)\gz(2n-2j)}{2^{2d-3}(2j+1)B_{2j}}.
\end{equation}
given by \cite[Theorem 1]{Hoffman2012}.
\end{proof}
For example, when $d<4$ we have
\begin{align}
A(2n,2)=&\frac{3}{4^n}\gz(2n)  \label{equ:Aga=2n2H}\\
A(2n,3)=&\frac{5}{4^n}\gz(2n)-\frac{2}{4^n}\gz(2)\gz(2n-2).\label{equ:Aga=2n3H}
\end{align}

The following lemma will be handy in the computation of
small depth cases.
\begin{lem}\label{lem:depthmixZetaDbl}
For every positive integer $n$ and nonnegative integer $r$ we define
\begin{equation*}
A^{(r)}_1(2n,2)=\sum_{j=1}^{n-1} j^r \gz(2j)\gz(\ol{2n-2j}).
\end{equation*}
Then we have
\begin{align}\label{equ:mixZetaDbl}
A^{(0)}_1(2n,2)=&\frac12\gz(2n)+n\gz(\ol{2n}),\\
\label{equ:wtmixZetaDbl}
A^{(1)}_1(2n,2)=&\frac{n}{2} \gz(2n)
+\frac{n(2n-1)}{4} \gz(\ol{2n})
-\frac{3}{2}\gz(\ol{2})\gz(\ol{2n-2}),\\
\label{equ:wtmixZetaTriple}
A^{(2)}_1(2n,2)=&\frac{n^2}{2} \gz(2n)
+\frac{n(2n-1)(4n-1)}{24} \gz(\ol{2n})
-\frac{4n+3}{4} \gz(\ol{2})\gz(\ol{2n-2}).
\end{align}
\end{lem}
\begin{proof}
We can obtain \eqref{equ:mixZetaDbl} by
setting $x=1/2$, $y=0$ and
replacing $n$ by $2n$ in \cite[(2.8)]{Nakamura2009}.

To prove \eqref{equ:wtmixZetaDbl} let $g(y)=ye^{y/2}/(e^y-1)$ and $h(y)=ye^{y}/(e^y-1)$.
Now consider the generating function
\begin{align*}
H(x,y)=&  \sum_{n=2}^{\infty}\frac{y^{2n}}{(2n)!}\sum_{j=1}^{n-1} x^{2j}B_{2j}B_{2n-2j}\Big(\frac12\Big) {2n\choose 2j} \\
=& \sum_{j=1}^{\infty} \left(\sum_{n=j+1}^{\infty} B_{2n-2j}\Big(\frac12\Big)\frac{y^{2n-2j}}{(2n-2j)!} \right) B_{2j}\frac{(xy)^{2j}}{(2j)!} \\
=&(g(y)-1) \left( \frac{xy}{e^{xy}-1}+\frac{xy}{2}-1 \right).
\end{align*}
An easy computation yields
\begin{align*}
\left.\frac{\partial}{\partial x}H(x,y)\right|_{x=1}
=& (g(y)-1) \left( \frac{y}{e^{y}-1}-\frac{y^2e^y}{(e^{y}-1)^2}
+\frac{y}{2} \right) \\
=& \frac{y^2}{8} g(y)-\frac{y^2}{2} g''(y)-yh'(y)+y/2.
\end{align*}
Comparing the coefficients of $y^{2n}$ we have
\begin{equation*}
\sum_{j=1}^{n-1} 2j \frac{B_{2j}}{(2j)!} \frac{B_{2n-2j} (1/2)} {(2n-2j)!}
=\frac{1}{8} \frac{B_{2n-2} (1/2)}{(2n-2)!}
- n(2n-1) \frac{B_{2n}(1/2)}{(2n)!}
- 2n\frac{B_{2n}}{(2n)!}
\end{equation*}
Multiplying by $(-1)^{n}(2\pi)^{2n}/8$ on both sides we get \eqref{equ:wtmixZetaDbl}.
Similarly it is straight-forward to verify that
\begin{align*}
 \left.\left(x\frac{\partial}{\partial x}\right)^2 H(x,y)\right|_{x=1}
=&\frac{7y^2}{24} g(y)+\frac{y^3}{12} g'(y) -\frac{y^2}2 g''(y)-\frac{y^3}{3} g'''(y)-yh'(y)-y^2h''(y)+\frac{y}{2}.
\end{align*}
Consequently
\begin{equation*}
\sum_{j=1}^{n-1} 4j^2 \frac{B_{2j}}{(2j)!} \frac{B_{2n-2j} (1/2)} {(2n-2j)!}
=\frac{4n+3}{24} \frac{B_{2n-2} (1/2)}{(2n-2)!}
- \frac{n(2n-1)(4n-1)}3 \frac{B_{2n}(1/2)}{(2n)!}
- 4n^2\frac{B_{2n}}{(2n)!}.
\end{equation*}
Multiplying by $(-1)^{n}(2\pi)^{2n}/16$ on both sides we get \eqref{equ:wtmixZetaTriple}.
This completes the proof of the lemma.
\end{proof}

In the next three sections we handle $A_\ga(2n,d)$ for $d=2,3,4$ respectively.
The main idea is to use stuffle relations (see \cite{Zeuler}) and \eqref{equ:Hoffman}
to express $A_\ga(2n,d)$ in terms
of sums similar to those in Lemma~\ref{lem:depthmixZetaDbl}.
For example, when there is only one alternating component we have
\begin{multline}  \label{equ:Aga=12nd}
 A_1(2n,d)=\sum_{\substack{j_1,\dots,j_d>0,\\ j_1+\cdots+j_d=n}}
\sum_{k=1}^d \gz(2j_1,\dots,2j_{k-1},\ol{2j_k},2j_{k+1},\dots,2j_d)  \\
= (-1)^{d-1}\binom{n-1}{d-1}\gz(\ol{2n})
+\sum_{k=1}^{d-1} (-1)^{k-1}\sum_{\ell=k}^{n+k-d} \binom{\ell-1}{k-1} \gz(\ol{2\ell})
A_0(2n-2\ell,d-k) .
\end{multline}

\begin{rem}
To compute $A_1(2n,d)$ for arbitrary $d$ we need to generalize Lemma~\ref{lem:depthmixZetaDbl} to find a formula of $A^{(r)}_1(2n,2)$
for all nonnegative integer $r$. We leave this to the interested reader.
In the last section we shall use symmetric function theory to find
a sum formula for $A_1(2n,d)$ for arbitrarily fixed $n$ and $d$.
\end{rem}

\section{Euler sums of depth two}
We call Euler sums of form $\gz(\ol{j_1},\dots,\ol{j_d})$ to be
of total alternating type. We first consider the restricted sums
involving only this type of Euler sums.
\begin{prop}\label{prop:Aga=22n2}
For all positive integers $n$ we have
\begin{equation}\label{equ:Aga=22n2}
A_2(2n,2)= \sum_{a,b>0,a+b=n}\gz(\ol{2a},\ol{2b})
= \frac14\gz(2n)+\frac12\gz (\ol{2n}).
\end{equation}
\end{prop}
\begin{proof}
Let $a,b>0$ such that $a+b=n$. Then by stuffle relation
and \eqref{equ:olzeta} we have
\begin{align*}
\gz(\ol{2a},\ol{2b})+\gz(\ol{2b},\ol{2a})=&\gz(\ol{2a})\gz(\ol{2b})-\gz(2a+2b)\\
=&(2^{1-2a}-1)(2^{1-2b}-1)\gz(2a)\gz(2b)-\gz(2n)\\
=&(2^{2-2n}-2^{1-2a}-2^{1-2b}+1)\gz(2a)\gz(2b)-\gz(2n).
\end{align*}
Therefore
\begin{align*}
2A_2(2n,2)=&(4^{1-n}+1)\sum_{a,b>0,a+b=n}\gz(2a)\gz(2b)
-\sum_{a,b>0,a+b=n}4^{1-a}\gz(2a)\gz(2b)-(n-1)\gz(2n)\\
=&\left(\frac{(4^{1-n}+1)(2n+1)}{2}-(n4^{1-n}+2)-(n-1)\right)\gz(2n)
= \left( \frac{2}{4^n}-\frac12 \right)\gz(2n).
\end{align*}
by \cite[(2.4)]{Nakamura2009} and the third displayed formula from bottom
on page 154 of \cite{Nakamura2009}. The proposition now follows easily.

Anther (maybe more elegant) proof is to set $x=y=1/2$ and
replace $n$ by $2n$ in \cite[(2.8)]{Nakamura2009}. We have
\begin{equation}\label{equ:depth2Bern12}
\sum_{a,b\ge 0,a+b=n}
{2n\choose 2a} B_{2a}\Big(\frac12\Big)B_{2b}\Big(\frac12\Big)=-(2n-1)B_{2n}.
\end{equation}
Notice that
$(-1)^{a-1} (2\pi)^{2a} B_{2a}(1/2)/(2(2a)!)=\gz(\ol{2a})$. Thus
\begin{equation}\label{equ:olZetaDbl}
\sum_{a,b>0,a+b=n}\gz(\ol{2a})\gz(\ol{2b})=\frac{2n-1}{2} \gz(2n)+\gz(\ol{2n}).
\end{equation}
Hence
\begin{equation*}
\sum_{a,b>0,a+b=n}\big(2\gz(\ol{2a},\ol{2b})+\gz(2n)\big)
=\frac{2n-1}{2} \gz(2n)+\gz(\ol{2n}).
\end{equation*}
This yields \eqref{equ:Aga=22n2} quickly.
\end{proof}

\begin{prop}\label{prop:Aga=1d=2}
For all positive integers $n$ we have
\begin{equation}\label{equ:Aga=1d=2}
A_1(2n,2)=\frac12\gz(2n)+\gz(\ol{2n}).
\end{equation}
\end{prop}
\begin{proof}
By \eqref{equ:Aga=12nd}
\begin{equation*}
A_1(2n,2)=\sum_{\substack{j+k=n\\ j,k>0}}\gz(2j)\gz(\ol{2k})-(n-1)\gz(\ol{2n}) .
\end{equation*}
The proposition follows from \eqref{equ:mixZetaDbl} quickly.
\end{proof}

\begin{prop}\label{prop:C2n2}
For all positive integers $n$ we have
\begin{equation}\label{equ:C2n2}
A(2n,2)= \sum_{a,b\in \barN,|a|+|b|=n}\gz(2a,2b)
=\frac32\gz(2n)+\frac32\gz (\ol{2n}).
\end{equation}
\end{prop}
\begin{proof}
This follows from \eqref{equ:Aga=2n2H} easily. But we may also prove it
by using the definition directly as follows:
\begin{align*}
A(2n,2)=&\sum_{a,b>0, a+b=n}\big(\gz(2a,2b)+\gz(2a,\ol{2b})
 +\gz(\ol{2a},2b)+\gz(\ol{2a},\ol{2b})\big) \\
=&\sum_{a,b>0, a+b=n}\big(\gz(2a,2b)+\gz(2a)\gz(\ol{2b})
 +\gz(\ol{2a},\ol{2b})-\gz(\ol{2n})\big).
\end{align*}
The proposition follows easily from \eqref{equ:GKZ}, \eqref{equ:Aga=22n2} and
\eqref{equ:mixZetaDbl}.
\end{proof}

\section{Euler sums of depth three}
We first consider the restricted sums involving only Euler sums of
total alternating type.
\begin{prop}\label{prop:Aga=32n3}
For all positive integers $n$ we have
\begin{equation}\label{equ:Aga=32n3}
A_3(2n,3)
=\frac{1}{8}\gz(2n)+\frac{1}2\gz(\ol{2n})
+\frac{1}{2}\gz(\ol{2})\gz(\ol{2n-2}).
\end{equation}
\end{prop}
\begin{proof}
Let $a=x+y+z$ and $f(t)=t e^{at}/(e^t-1)$. Then
\begin{multline*}
    \frac{te^{xt}}{e^t-1}\frac{t e^{yt}}{e^t-1}\frac{t e^{zt}}{e^t-1}
= \frac{t^3 e^{at}}{(e^t-1)^3}\\
=\frac12\Big\{\big[(a-1)(a-2)t^2+(2a-3)t+2\big]f(t)
-\big[(2a-3)t^2+2t\big]f'(t)+t^2f''(t)\Big\}  .
\end{multline*}
Comparing the coefficients of $t^n$ we get
\begin{multline*}
 \sum_{\substack{j+k+\ell=n\\ j,k,\ell\ge 0}}
{n\choose j,k,\ell} B_j(x)B_k(y)B_\ell(z) \\
=\frac12\Big\{ n(n-1)(a-1)(a-2)B_{n-2}(a)-n(n-2)(2a-3)B_{n-1}(a)
 + (n-1)(n-2)B_n(a) \Big\}  .
\end{multline*}
Setting $x=y=z=1/2$ (so $a=3/2$) and $n\to 2n$ we get
\begin{multline*}
3\sum_{\substack{j+k=n\\ j,k\ge 0}}
{2n\choose 2j} B_{2j}\Big(\frac12\Big)B_{2k}\Big(\frac12\Big)
+ \sum_{\substack{j+k+\ell=n\\ j,k,\ell>0}}
{2n\choose 2j,2k,2\ell} B_{2j}\Big(\frac12\Big)B_{2k}\Big(\frac12\Big) B_{2\ell}\Big(\frac12\Big) \\
=3B_{2n}\Big(\frac12\Big)+\frac12\Big\{-n(2n-1)B_{2n-2}\Big(\frac12\Big)/2
 + (2n-1)(2n-2)B_{2n}\Big(\frac12\Big) \Big\},
\end{multline*}
since $B_m(x)$ is periodic function with period 1 for all $m\ge 0$
and $B_{2n-1}\big(\frac12\big)=0$ for all $n\ge 1$.
Multiplying by $(-1)^{n-1}(2\pi)^{2n}/8$ on both sides,
using \eqref{equ:depth2Bern12} and simplifying we have
\begin{equation*}
\sum_{\substack{j+k+\ell=n\\ j,k,\ell>0}}
\gz(\ol{2j})\gz(\ol{2k})\gz(\ol{2\ell})
=\frac{\pi^2}{8}\gz(\ol{2n-2})
 + \frac{2n^2-3n+4}4 \gz(\ol{2n}) +\frac{6n-3}4 \gz(2n)  .
\end{equation*}
On the other hand, by stuffle relations
\begin{align*}
 \sum_{\substack{j+k+\ell=n\\ j,k,\ell>0}}\gz(\ol{2j})\gz(\ol{2k})\gz(\ol{2\ell})
=&\sum_{\substack{j+k+\ell=n\\ j,k,\ell>0}}\Big(2\gz(\ol{2j},\ol{2k})\gz(\ol{2\ell})
+\gz(2j+2k)\gz(\ol{2\ell})\Big) \\
=& \sum_{\substack{j+k+\ell=n\\ j,k,\ell>0}}\Big(6\gz(\ol{2j},\ol{2k},\ol{2\ell})
+3\gz(2j+2k)\gz(\ol{2\ell})-2\gz(\ol{2n})\Big) \\
=& 6\sum_{\substack{j+k+\ell=n\\ j,k,\ell>0}} \gz(\ol{2j},\ol{2k},\ol{2\ell})
+3\sum_{\substack{j+k=n\\ j,k>0}}(j-1)\gz(2j)\gz(\ol{2k})-\binom{n-1}{2}\gz(\ol{2n}) .
\end{align*}
The proposition follows easily.
\end{proof}

\begin{prop}\label{prop:Ad=3}
For every positive integer $n$ we have
\begin{align} \label{equ:Aga=1d=3}
A_1(2n,3)=&\frac{7}8\gz(2n)+\gz(\ol{2n})\\
\label{equ:Aga=2d=3}
A_2(2n,3)
=&\frac{7}8\gz(2n)+\gz(\ol{2n}).
\end{align}
\end{prop}
\begin{proof}
By \eqref{equ:Aga=12nd} and \eqref{equ:GKZ} we have
\begin{align*}
\text{LHS of }\eqref{equ:Aga=1d=3}
=&\frac34\sum_{j=1}^{n-2} \gz(\ol{2j})\gz(2n-2j)
-\sum_{j=1}^{n-1} (j-1)\gz(\ol{2j})\gz(2n-2j)+{n-1\choose 2}\gz(\ol{2n}).
\end{align*}
Thus \eqref{equ:Aga=1d=3} follows
from Lemma~\ref{lem:depthmixZetaDbl} easily. Similarly by stuffle relations and \eqref{equ:Aga=22n2}
\begin{align} \notag
\text{LHS of }\eqref{equ:Aga=2d=3}
=&\sum_{\substack{j+k+\ell=n\\ j,k,\ell> 0}}
\Big(\gz(2j)\gz(\ol{2k},\ol{2\ell})-\gz(\ol{2j+2k})\gz(\ol{2\ell})+\gz(2n)\Big)\\
=&\frac14\sum_{j=1}^{n-1} \gz(2j) \gz(2n-2j)
+\frac12\sum_{j=1}^{n-1}  \gz(2j)\gz(\ol{2n-2j}) \label{equ:needcancel}\\
 & \hskip2cm -\sum_{j=1}^{n-1} (j-1)\gz(\ol{2j})\gz(\ol{2n-2j})+{n-1\choose 2}\gz(2n). \notag
\end{align}
Here we have used the coincidence that when $j=n-1$ the corresponding two terms of
the two sums in \eqref{equ:needcancel} cancel each other.
By changing index $j\to n-j$ and adding it back we see that
$$\sum_{j=1}^{n-1} (j-1)\gz(\ol{2j})\gz(\ol{2n-2j})
=\frac{n-2}2\sum_{j=1}^{n-1} \gz(\ol{2j})\gz(\ol{2n-2j}).$$
Hence  \eqref{equ:Aga=2d=3} follows from
\cite[(2.4)]{Nakamura2009}, \eqref{equ:mixZetaDbl} and \eqref{equ:olZetaDbl}.
\end{proof}

Combining \cite[Theorem 1]{ShenCai2012}, \eqref{equ:Aga=32n3} and Proposition~\ref{prop:Ad=3} we can prove the next result immediately.
Of course this also follows from \eqref{equ:Aga=2n3H} easily.
\begin{prop}\label{prop:C2n2}
For all positive integers $n$ we have
\begin{equation*}
A(2n,3)= \frac52\gz(2n)+\frac52\gz (\ol{2n})
+\frac12\gz (\ol{2})\big(\gz (\ol{2n-2})+\gz (2n-2)\big).
\end{equation*}
\end{prop}

\section{Euler sums of depth four}
One needs the following Lemma in depth four.
\begin{lem}\label{lem:depthmixZetaDbl2ndmoment}
For every positive integer $n$ and nonnegative integer $r$ we have
\begin{multline}
\label{equ:depthmixZetaDbl2ndmoment}
L_2(n)=\sum_{j=1}^{n-1} j^2 \gz(\ol{2j})\gz(\ol{2n-2j})\\
=\frac{n(2n-1)(4n-1)}{24}\gz(2n)
+\frac{2n-3}{4}\gz(2)\gz(2n-2)+\frac{n^2}{2}\gz(\ol{2n}).
\end{multline}
\end{lem}
\begin{proof}
Let $g(y)=ye^{y/2}/(e^y-1)$ and $h(y)=y/(e^y-1)$.
Now consider the generating function
\begin{align*}
G(x,y)=&  \sum_{n=0}^{\infty}\frac{y^{2n}}{(2n)!}\sum_{j=0}^{n}
x^{2j}B_{2j}\Big(\frac12\Big) B_{2n-2j}\Big(\frac12\Big) {2n\choose 2j} \\
=& \sum_{j=0}^{\infty} \left(\sum_{n=j}^{\infty} B_{2n-2j}\Big(\frac12\Big)\frac{y^{2n-2j}}{(2n-2j)!} \right)
B_{2j}\Big(\frac12\Big) \frac{(xy)^{2j}}{(2j)!} \\
=&g(y)g(xy).
\end{align*}
Thus
\begin{equation*}
 \left.\left(x\frac{\partial}{\partial x}\right)^2 G(x,y)\right|_{x=1}
= -\frac13 y^3 h'''(y)-\frac12 y^2 h''(y) +\frac1{12} y^3 h'(y)-\frac1{12} y^2 h(y).
\end{equation*}
Consequently
\begin{equation*}
\sum_{j=1}^{n-1} 4j^2 \frac{B_{2j} (1/2)}{(2j)!}\frac{B_{2n-2j} (1/2)} {(2n-2j)!}
=- \frac{n(2n-1)(4n-1)}3 \frac{B_{2n}}{(2n)!}
+\frac{2n-3}{12} \frac{B_{2n-2}}{(2n-2)!}.
\end{equation*}
Multiplying by $(-1)^{n}(2\pi)^{2n}/16$ on both sides we get \eqref{equ:depthmixZetaDbl2ndmoment}.
This completes the proof of the lemma.
\end{proof}

\begin{prop}\label{prop:Aga=gaAlld4}
For every positive integer $n$ we have
\begin{alignat}{4}  \label{equ:Aga=1d4}
A_1(2n,4)=&\frac{19}{16} \gz(2n)+&\gz(\ol{2n})&-&\frac18\gz(2)\gz(2n-2)\\
 \label{equ:Aga=2d4}
A_2(2n,4)=&\frac{57}{32} \gz(2n)+&\frac32 \gz(\ol{2n})&-&\frac3{16}\gz(2) \gz(2n-2)  \\
 \label{equ:Aga=3d4}
A_3(2n,4)=&\frac{11}{16}\gz(2n)+&\frac32 \gz(\ol{2n}) &-&\frac12\gz(2)\gz(\ol{2n-2}) \\
\label{equ:Aga=42nd4}
A_4(2n,4)=&\frac{11}{64}\gz(2n)+&\frac38 \gz(\ol{2n}) &-&\frac18\gz(2)\gz(\ol{2n-2}),
\end{alignat}
and
\begin{equation}\label{equ:mixZeta4All}
A(2n,4)=\frac{35}{4^{n+1}} \gz(2n)-\frac5{4^n}\gz(2)\gz(2n-2).
\end{equation}
\end{prop}
\begin{proof}
One may use the same method as in the last section. We will only prove \eqref{equ:Aga=2d4}
which is not covered by the last two sections of this paper and is
the most complicated case in depth 4. By definition
\begin{multline*}
2 A_2(2n,4)= \sum_{\substack{j_1+j_2+j_3+j_4=n\\ j_1,j_2,j_3,j_4>0}}
\gz(\ol{2j_1})\Big( \gz(\ol{2j_2},2j_3,2j_4)+\gz(2j_2,\ol{2j_3},2j_4)
+\gz(2j_2,2j_3,\ol{2j_4}) \Big)\\
 -X_2(n)-X_3(n)
\end{multline*}
where
\begin{align*}
X_2=& \sum_{\substack{j_1+j_2+j_3+j_4=n\\ j_1,j_2,j_3,j_4>0}}
\Big( \gz(2j_1+2j_2,2j_3,2j_4)+\gz(2j_1,2j_2+2j_3,2j_4)
+\gz(2j_1,2j_2,2j_3+2j_4) \Big) \\
=&\sum_{\substack{j+k+\ell=n\\ j,k,\ell> 0}}
 (j-1)\Big(\gz(2j)\gz(2k,2\ell)-\gz(2j+2k)\gz(2\ell)+\gz(2n) \Big) \\
=&\sum_{j=1}^{n-2} (j-1)\gz(2j)A_0(2n-2j,2)
-\sum_{k=1}^{n-1}\binom{k-1}{2}\gz(2k)\gz(2n-2k)+\binom{n-1}{3}\gz(2n)   \\
\end{align*}
and similarly
{\allowdisplaybreaks
\begin{align*}
X_3=&\sum_{\substack{j_1+j_2+j_3+j_4=n\\ j_1,j_2,j_3,j_4>0}}
 \Big( \gz(\ol{2j_1+2j_2},\ol{2j_3},2j_4)+\gz(\ol{2j_1},\ol{2j_2+2j_3},2j_4)
+\gz(2j_1,\ol{2j_2},\ol{2j_3+2j_4}) \\
&+\gz(2j_1,\ol{2j_2+2j_3},\ol{2j_4})
+\gz(\ol{2j_1+2j_2},2j_3,\ol{2j_4}) +\gz(\ol{2j_1},2j_2,\ol{2j_3+2j_4}) \Big) \\
=&\sum_{\substack{j+k+\ell=n\\ j,k,\ell> 0}}
 (j-1)\Big\{\gz(\ol{2j})\Big(\gz(\ol{2k},2\ell)+\gz(2k,\ol{2\ell})\Big)
- \gz(2j+2k,2\ell)-\gz(2\ell,2j+2k)  \\
& \hskip7cm -  \gz(\ol{2j+2k},\ol{2\ell})-\gz(\ol{2\ell},\ol{2j+2k})\Big\}
\end{align*}
}
Hence
\begin{align*}
2A_2(2n,4)=  \sum_{\ell=1}^{n-1}\frac{(\ell-1)(\ell-2)}{2}
 \Big(2\gz(2\ell) \gz(2n-2\ell)+\gz(\ol{2\ell}) \gz(\ol{2n-2\ell})\Big)
 -3\binom{n-1}{3}\gz(\ol{2n})   \\
 +\sum_{\ell=1}^{n-3} \gz(\ol{2\ell}) A_1(2n-2\ell,3)
 -\sum_{\ell=1}^{n-2}(\ell-1)\Big\{ \gz(2\ell) A_0(2n-2\ell,2)
 +\gz(\ol{2\ell}) A_1(2n-2\ell,2)\Big\}.
\end{align*}
We now can finish the proof of \eqref{equ:Aga=2d4}
by using \eqref{equ:Hoffman}, Proposition~\ref{prop:Aga=1d=2},  Proposition~\ref{prop:Ad=3}, Lemma~\ref{lem:depthmixZetaDbl}, Lemma~\ref{lem:depthmixZetaDbl2ndmoment} and the following
identities (see \cite[(2.4)]{Nakamura2009} and \cite[(11.10)]{Shimura2007})
\begin{align*}
 A^{(0)}_0(2n,2)=\sum_{\ell=1}^{n-1} \gz(2\ell) \gz(2n-2\ell) =&\frac{2n+1}{2} \gz(2n), \\
 A^{(1)}_0(2n,2)=\sum_{\ell=1}^{n-1} \ell \gz(2\ell) \gz(2n-2\ell) =&\frac{n(2n+1)}{4} \gz(2n), \\
 A^{(2)}_0(2n,2)=\sum_{\ell=1}^{n-1} \ell^2\gz(2\ell) \gz(2n-2\ell) =&
 \frac{n(8n^2+6n+1)}{24} \gz(2n)-\frac{2n-3}{2} \gz(2)\gz(2n-2).
\end{align*}

All the other formulas in the proposition can be proved similarly. In particular,
we have two different proofs of \eqref{equ:mixZeta4All},
one by Proposition~\ref{prop:Aga=2nd} and the other by adding up
all $A_\ga(2n,4)$ altogether for $\ga=0,\dots,4$, which provides us some confidence
that all the formulas are correct.
We leave the details to the interested reader.
\end{proof}

\section{Euler sums of total alternating type}
As in \cite{Hoffman2012} we define an evaluation homomorphism $\evalE:\Sym\to\R$ such that
$\evalE(x_j)=1/j^2$ for all $i\ge 1$. In proving \cite[Lemma 1]{Hoffman2012} Hoffman showed that
\begin{equation}\label{equ:HoffmanEval}
\evalE(E(u))=\frac{\sinh(\pi\sqrt u)}{\pi\sqrt u}.
\end{equation}
Define for all positive integers $d\le n$
$$M^{\rm tot}_{n,d}=\sum_{j_1<\cdots<j_d} \sum_{n_1+\cdots+n_d=n} \sum_{k=1}^d
(-1)^{j_1+\cdots +j_d}x_{j_1}^{n_1}\cdots x_{j_d}^{n_d}.$$
Clearly we have $\evalE(M^{\rm tot}_{n,d})=A_d(2n,d)$.
We first determine the generating function of $A_d(2n,d)$ completely.
\begin{thm} \label{thm:gfunTot}
Define the generating function
$$\psi_{\rm tot}(u,v)= 1+\sum_{n\ge d\ge 1} A_d(2n,d) u^n v^d.$$
Then we have
\begin{equation}\label{equ:psitot}
\psi_{\rm tot}(u,v)= \frac{\sin( \pi \sqrt{(1-v)u}/2) }{ \sqrt{1-v}\sin(\pi\sqrt u/2)}
 \cdot\frac{\cos( \pi \sqrt{(v+1)u}/2)}{ \cos(\pi\sqrt u/2)}.
\end{equation}
\end{thm}
\begin{proof}
It is easy to see that the generating function of $M^{\rm tot}_{n,d}$ is given by
\begin{equation*}
\calF_{\rm tot}(u,v)=1+\sum_{n\ge d\ge 1} M^{\rm tot}_{n,d} u^n v^d
=\prod_{j=1}^\infty  (1+(-1)^j(vux_j+vu^2x_j^2+\cdots)).
\end{equation*}
Hence
\begin{align*}
 \calF_{\rm tot}(u,v)=& \prod_{j=1}^\infty
 \left(1+ \frac{(-1)^j vu x_j}{1-ux_j}\right)\\
=&  \prod_{j=1}^\infty\frac{  ( 1+(v-1)u x_{2j}) ( 1-(v+1)u x_{2j-1})} {1-ux_j}.
\end{align*}
Define the even and odd parts of $E(u)$ by
\begin{equation*}
  E^e(u)= \prod_{j=1}^\infty (1+u x_{2j}), \quad\text{and}\quad
 E^o(u)=\prod_{j=1}^\infty (1+u x_{2j-1}),
\end{equation*}
respectively. Then we have
\begin{equation}\label{equ:Ftot}
 \calF_{\rm tot}(u,v)=E^e((v-1)u)E^o(-(v+1)u)E(-u)^{-1}.
\end{equation}
Clearly
\begin{align}
f(u):=\evalE( E^e(u))=&\prod_{j=1}^\infty \left(1+\frac{u}{(2j)^2}\right)
=\frac{\sinh( \pi \sqrt{u}/2) }{ \pi \sqrt{u}/2}, \label{equ:evalEvalEven}\\
g(u):=\evalE( E^o(u))=&\prod_{j=1}^\infty \left(1+\frac{u}{(2j-1)^2}\right)
=\frac{\evalE( E(u))}{\evalE( E^e(u))}= \cosh( \pi \sqrt{u}/2).\label{equ:evalEvalodd}
\end{align}
The theorem follows from evaluating \eqref{equ:Ftot} using the above and
\eqref{equ:HoffmanEval}.
\end{proof}

\begin{thm}\label{thm:Aga=ddArbitray}
For every positive integers $d\le n$ we have
\begin{align*} 
& A_d(2n,d)=\sum_{j=0}^{\lfloor(d-1)/2 \rfloor}\frac{(-1)^j4^{j-n-d+1}}{(2j+1)!}
 \binom{2d-2j-1}{d} \gz(2n-2j)\pi^{2j} \\
+&\sum_{c=1}^d \sum_{j=0}^{\lfloor(c-1)/2 \rfloor}\sum_{k=0}^{\lfloor(d-c)/2 \rfloor}
\frac{(-1)^{c+j+k} (1-4^{j+k-n})}{c (2j)!(2k)!4^{d-1}} \binom{2c-2j-2}{c-1} \binom{2d-2c-2k}{d-c} z_{n,j,k} \\
+&\sum_{c=1}^d \sum_{j=1}^{\lfloor c/2 \rfloor} \sum_{k=0}^{\lfloor(d-c-1)/2 \rfloor}
\frac{(-1)^{c+j+k} 4^{j+k-n-d+1}}{c (2j-1)!(2k+1)!}
\binom{2c-2j-1}{c-1} \binom{2d-2c-2k-1}{d-c}
z_{n,j,k},
\end{align*}
where $z_{n,j,k}=\gz(2n-2j-2k)\pi^{2j+2k}.$
\end{thm}
\begin{proof}
Notice that the first factor of $\psi_{\rm tot}(u,v)$ in \eqref{equ:psitot}
is exactly the generating function for $A_0(2n,d)$ (i.e. the restricted sum
of multiple zeta values of fixed weight $2n$ and depth $d$). This is given by
\eqref{equ:Hoffman} found by Hoffman. For the second factor we
see from the proof of \cite[Theorem 1.1]{ZTn} that
\begin{align*}
\frac{\cos(\pi\sqrt{(1-v)u}/2)}{ \cos(\pi\sqrt{u}/2)}
=1+\sum_{d\ge 1} v^d &\left(  -\frac{\pi^2 u}{d} \sum_{j=0}^{\lfloor \frac{d-2}2\rfloor}\frac{(-\pi^2u)^j}
{2^{2d}(2j+1)!}\binom{2d-2j-3}{d-1}  \right.  \\
&  \left. +\frac{\pi \sqrt{u}}{d}  \tan(\pi \sqrt{u}/2) \sum_{j=0}^{\lfloor \frac{d-1}2\rfloor}\frac{(-\pi^2u)^j}{2^{2d}(2j)!}
\binom{2d-2j-2}{d-1}\right)
\end{align*}
Changing $v$ to $-v$ and combining with \cite[Lemma 3]{Hoffman2012}
we get
\begin{align*}
\psi_{\rm tot}(u,v)=\left\{ \sum_{d\ge 0}  v^d  \right.
 & \left( - \pi\sqrt u \cot(\pi\sqrt u/2) \sum_{k=0}^{\lfloor \frac{d-1}2\rfloor}\frac{(-\pi^2 u)^k}
{2^{2d}(2k+1)!}\binom{2d-2k-1}{d}\right.  \\
&  \hskip4cm \left.\left. +\sum_{k=0}^{\lfloor \frac{d}2\rfloor}\frac{(-\pi^2 u)^k}{2^{2d}(2k)!}
\binom{2d-2k}{d}\right) \right\}\\
\times\left\{1+\sum_{c\ge 1} (-v)^c \right. &\left( \frac{1}{c} \sum_{j=1}^{\lfloor \frac{c}2\rfloor}\frac{(-\pi^2u)^j}
{2^{2c}(2j-1)!}\binom{2c-2j-1}{c-1}  \right. \\
&  \left.\left.  +\frac{\pi \sqrt{u}}{c}  \tan(\pi \sqrt{u}/2) \sum_{j=0}^{\lfloor \frac{c-1}2\rfloor}\frac{(-\pi^2u)^j}{2^{2c}(2j)!}
\binom{2c-2j-2}{c-1}\right) \right\}
\end{align*}
Finally we can extract the coefficient of $v^du^n$
using the well-know identities
\begin{align*}
\pi\sqrt u \cot(\pi\sqrt u/2)=& -4\sum_{m=0}^\infty \frac{\gz(2m)}{4^m} u^m,\quad
\pi\sqrt u \tan (\pi\sqrt u/2)= 4\sum_{m=0}^\infty (4^m-1) \frac{\gz(2m)}{4^m} u^m.
\end{align*}
This completes the proof theorem.
\end{proof}

By setting $d=2,3,4$ we get \eqref{equ:Aga=22n2}, \eqref{equ:Aga=32n3} and \eqref{equ:Aga=42nd4}, respectively.

\section{Euler sums with one alternating component}
Define the generating function
\begin{equation*}
\psi_1(u,v)=  \sum_{n\ge d\ge 1} A_1(2n,d) u^n v^d.
\end{equation*}
Using symmetric functions we can find an explicit expression of this function.
\begin{thm}\label{thm:genFunPsi1}
We have
\begin{equation} \label{equ:Evafunc}
\psi_1(u,v)= \frac{v\sin(\pi \sqrt{(1-v)u})}{ \sqrt{1-v} \sin(\pi \sqrt{u})}
\left(\frac{1}{2(1-v)}- \frac{\pi\sqrt{u}}{2\sqrt{1-v} \sin( \pi\sqrt{(1-v)u})} \right).
\end{equation}
\end{thm}
\begin{proof}
Define for all positive integers $d\le n$
$$M^{1}_{n,d}=\sum_{j_1<\cdots<j_d} \sum_{n_1+\cdots+n_d=n} \sum_{k=1}^d
(-1)^{j_k}x_{j_1}^{n_1}\cdots x_{j_d}^{n_d}.$$
Clearly we have $\evalE(M^{1}_{n,d})=A_1(2n,d)$. Moreover the generating
function of $M^{1}_{n,d}$ is given by
$$\calF_1(u,v)=\sum_{k=1}^\infty (-1)^k(vux_k+vu^2x_k^2+\cdots)
\prod_{j\ne k} (1+vux_i+vu^2x_i^2+\cdots)=\sum_{n\ge d\ge 1} M^{1}_{n,d} u^n v^d.$$
It is straight-forward to see that
$$\calF_1(u,v)=E((v-1)u)H(u)\sum_{k=1}^\infty  \frac{(-1)^k vu x_k}{1+(v-1)ux_k}.$$
Therefore
\begin{equation*}
\psi_1(u,v)=\evalE(\calF_1(u,v))=\evalE(E((v-1)u)H(u))vu\left(\frac{f'((v-1)u)}{f((v-1)u)}- \frac{g'((v-1)u)}{g((v-1)u)} \right),
\end{equation*}
where $f(u)$ and $g(u)$ are defined by \eqref{equ:evalEvalEven} and
\eqref{equ:evalEvalodd}, respectively. They satisfy
\begin{equation*}
\frac{f'(u)}{f(u)}=  \coth( \pi \sqrt{u}/2) \frac{\pi}{4\sqrt{u}}-\frac{1}{2(v-1)u},  \quad
\frac{g'(u)}{g(u)}=  \tanh( \pi \sqrt{u}/2) \frac{\pi}{4\sqrt{u}}.
\end{equation*}
The equation \eqref{equ:Evafunc} follows immediately from the identity
$$\frac{\coth(\gt/2)-\tanh(\gt/2)}{2}
=\frac{\cosh^2(\gt/2)-\sinh^2(\gt/2)}{2\sinh(\gt/2)\cosh(\gt/2)}
=\frac{1}{\sinh(\gt)}.$$
The theorem is now proved.
\end{proof}

Taking $v\to 1$ Theorem \ref{thm:genFunPsi1} we have by L'H\^opital's Rule
\begin{multline*}
 \psi_1(u,1)=-\frac{\pi^2 u }{12}\pi\sqrt{u} \csc(\pi \sqrt{u})\\
 =-\frac{\pi^2 u }{12}
 \sum_{n=0}^\infty \frac{2(2^{2n-1}-1)(-1)^{n-1}B_{2n}\pi^{2n}}{(2n)!} u^n
 =\frac{\pi^2 u }{6}\sum_{n=0}^\infty  \gz(\ol{2n}) u^n.
\end{multline*}
Therefore we get
\begin{cor} \label{cor:Aga=1dAll}
For every positive integer $n$ we have
\begin{equation}\label{equ:Aga=1dAll}
\sum_{d=1}^n A_1(2n,d)=\gz(2)\gz(\ol{2n-2}).
\end{equation}
\end{cor}

Finally we are ready to prove the restricted sum formula for all Euler sums of exactly one
alternating component with fixed weight and depth at even arguments.
\begin{thm}\label{thm:Aga=1dArbitray}
For every positive integers $d\le n$ we have
\begin{equation}\label{equ:Aga=1dArbitray}
A_1(2n,d)=\gz(\ol{2n})-\sum_{j=0}^{\lfloor(d-2)/2 \rfloor}\sum_{k=0}^{2j+1} \frac{(-1)^{d+j+k}}{(2j+1)!}\binom{2j+1}{k}
\binom{(k-3)/2}{d-1} \gz(2n-2j)\pi^{2j}.
\end{equation}
\end{thm}
\begin{proof}
We expand \eqref{equ:Evafunc} as follows:
\begin{align*}
\psi_1(u,v))=& \frac{v}{2(1-v)^{3/2}}
\frac{\sin(\pi\sqrt{(1-v)u})}{\sin(\pi \sqrt{u})}
-\frac{v\pi\sqrt{u}}{2(1-v)\sin(\pi \sqrt{u})}  \\
=& \frac{v\pi \sqrt{u}}{2 \sin(\pi \sqrt{u})} \sum_{j=1}^\infty
\frac{(-1)^j  \pi^{2j} u^j}{(2j+1)!} (1-v)^{j-1}
=\sum_{d=1}^\infty v^d g_d(\pi^2 u)
\end{align*}
where for all $d\ge 0$
\begin{align*}
g_{d+1}(u)=&(-1)^d\frac{\sqrt u}{2\sin\sqrt u}
\sum_{j\ge d+1}\frac{(-1)^{j}u^{j}}{(2j+1)!}\binom{j-1}{d}\\
=&\frac{u\sqrt u}{2\sin\sqrt u}\cdot\frac{(-u)^{d}}{d!}\cdot
D^d \left(\frac{\sin\sqrt u}{u\sqrt u}-\frac1u\right) \\
=&\frac{u\sqrt u}{2\sin\sqrt u}\cdot\frac{(-u)^{d}}{d!}\cdot
D^d \left(\frac{\sin\sqrt u}{u\sqrt u}\right)-\frac{ \sqrt u}{2\sin\sqrt u}.
\end{align*}
Let $h(u)=\sin\sqrt u/(u\sqrt u).$ We want to find polynomials $\tP_d(u)$
and  $\tQ_d(u)$ such that for all $d\ge 0$
\begin{equation*}
2d! \left(g_{d+1}(u)+\frac{ \sqrt u}{2\sin\sqrt u}\right)=\tP_d(u)\sqrt u\cot\sqrt u+\tQ_d(u),
\end{equation*}
with initial conditions $\tP_0(u)=0$ and $\tQ_0(u)=1$. Thus
\begin{equation*}
h^{(d)}(u)= (-1)^d u^{-d-1} \tP_d(u)\cos\sqrt u +(-1)^d u^{-d-3/2} \tQ_d(u)\sin\sqrt u,
\end{equation*}
from which we get the following recursive system of differential equations: $\forall d\ge 0$
\begin{equation}\label{equ:DEsystem}
\left\{
\aligned
\tP_{d+1}(x)=&  (d+1)\tP_d(x)-x\tP'_d(x)- \frac12 \tQ_d(x) \\
\tQ_{d+1}(x)=& \frac{2d+3}{2}\tQ_d(x)-x\tQ'_d(x)+\frac{x}2 \tP_d(x)
\endaligned
\right.
\end{equation}
Now define
\begin{align*}
p(x,y)=&\sum_{d=0}^\infty \tP_d(x) \frac{y^d}{d!}=\sum_{j=0}^\infty p_j(y) x^j,\\
q(x,y)=&\sum_{d=0}^\infty \tQ_d(x) \frac{y^d}{d!}=\sum_{j=0}^\infty q_j(y) x^j.
\end{align*}
Then the system \eqref{equ:DEsystem} translates into the following: $\forall j\ge1$
\begin{equation}\label{equ:DEsystemNew}
\left\{
\aligned
(1-y)p'_j(y)+(j-1)p_j(y)\, =\, &-\frac12 q_j(y) \\
(1-y)q'_j(y) +\Big(j-\frac32\Big )q_j(y)=\, &\frac12 p_{j-1}(y)
\endaligned
\right.
\end{equation}
Moreover, since
$$\tQ_d(0)= \frac{2d+1}{2}\tQ_{d-1}(0)=\cdots=\frac{(2d+1)!!}{2^d}$$
we see that
$$q_0(y)=q(0,y)=\sum_{d=0}^\infty \tQ_d(0) \frac{y^d}{d!}=(1-y)^{-3/2}.$$
After solving \eqref{equ:DEsystemNew} recursively for the first few cases
one can easily observe some obvious pattern. Then it is not hard to verify that
\begin{align*}
p_j(y)=& \frac{(-1)^{j+1}}{(2j+1)!}\sum_{k=0}^{2j+1} (-1)^k \binom{2j+1}{k} (1-y)^{(k-3)/2},\\     q_j(y)=& \frac{(-1)^{j}}{(2j)!}\sum_{k=0}^{2j} (-1)^k \binom{2j}{k} (1-y)^{(k-3)/2},
\end{align*}
from which we can find immediately:
\begin{align*}
\tP_{d}(x)=& d!\sum_{j=0}^{\lfloor(d-1)/2 \rfloor}\sum_{k=0}^{2j+1} \frac{(-1)^{d+j+k+1}}{(2j+1)!}\binom{2j+1}{k}
\binom{(k-3)/2}{d}x^j,\\
\tQ_{d}(x)=& d!\sum_{j=0}^{\lfloor d/2 \rfloor}\sum_{k=0}^{2j} \frac{(-1)^{d+j+k}}{(2j)!}\binom{2j}{k}
\binom{(k-3)/2}{d}x^j.
\end{align*}
By \eqref{equ:cotFormula} we see that
\begin{multline*}
\left(g_{d+1}(u)+\frac{\sqrt u}{2\sin\sqrt u}\right)=
\sum_{\ell=0}^\infty \frac{\gz(2\ell)}{\pi^{2\ell}} u^\ell \sum_{j=0}^{\lfloor(d-1)/2 \rfloor}\sum_{k=0}^{2j+1} \frac{(-1)^{d+j+k}}{(2j+1)!}\binom{2j+1}{k}
\binom{(k-3)/2}{d}u^j \\
+\text{term of degree}<d.
\end{multline*}
Replacing $d$ by $d-1$, $u$ by $\pi^2 u$ and
considering the coefficient of $u^n$ we get \eqref{equ:Aga=1dArbitray}.
This finishes the proof of the theorem.
\end{proof}

By taking $d=2,3,4$ in Theorem~\ref{thm:Aga=1dArbitray} we get
\eqref{equ:Aga=1d=2}, \eqref{equ:Aga=1d=3} and \eqref{equ:Aga=1d4}, respectively.

\begin{rem}
Notice that the coefficients for $x^j$ in $\tP_{d}(x)$ and $\tQ_{d}(x)$
are given by sums instead of just one term. Using WZ method
we can show that no closed formulas in terms of $j$ can be found
in the sense of \cite[\S 8.7]{WZ}.
\end{rem}

We conclude the paper by pointing out that
it is possible to carry out the above procedure
to compute $A_\ga(2n,d)$ for each specific $d$ and $\ga$
when $\ga$ is either close to $0$ or close to $d$. But
when $\ga$ moves closer to $d/2$ it seems that the method used in \cite{Hoffman2012,ZTn}
and this paper becomes too unwieldy to apply.

\end{document}